\documentclass[a4paper, 11pt]{scrartcl}

\usepackage[english]{babel}

\usepackage[colorlinks=true, pdfstartview=FitV, linkcolor=black, citecolor=red, breaklinks=true, urlcolor=red]{hyperref}

\usepackage{stmaryrd}
\usepackage{amssymb,amsmath,amsthm} 
\usepackage{graphicx,color,mathrsfs}

\usepackage{todonotes}

\newtheorem{theorem}{Theorem}

\newtheorem{lemma}[theorem]{Lemma}

\newtheorem{definition}[theorem]{Definition}

\newtheorem{proposition}[theorem]{Proposition}

\theoremstyle{definition}

\newtheorem{remark}[theorem]{Remark}

\newcommand{\R}{\mathbb{R}}
\newcommand{\N}{\mathbb{N}}

\newcommand{\vp}{\varepsilon}

\newcommand{\beq}{\begin{equation}}
\newcommand{\eeq}{\end{equation}}
\newcommand{\bpm}{\begin{pmatrix}}
\newcommand{\epm}{\end{pmatrix}}

\DeclareMathOperator{\diag}{diag}

\title{Stabilisation of the complex double integrator  by means of a saturated linear feedback}

\author{{Yacine Chitour}
\thanks{Y. Chitour is with
Laboratoire des Signaux et Syst\`emes (L2S, UMR CNRS 8506), Universit\'e Paris-Saclay - CNRS - CentraleSupelec, 3, rue Joliot Curie, 91192, Gif-sur-Yvette, France,
\texttt{yacine.chitour@l2s.centralesupelec.fr}}
 \footnote{This research was partially supported by the iCODE Institute, research project of the IDEX Paris-Saclay, and by the Hadamard Mathematics LabEx (LMH) through the grant number ANR-11-LABX-0056-LMH in the ``Programme des Investissements d'Avenir''.}
}

\begin{document}

\maketitle
\begin{center}
To E. D. Sontag, for his $70$th birthday.
\end{center}

\vspace{2cm}

\begin{abstract} Consider the saturated complex double integrator, i.e., the linear control 
system $\dot x=Ax+B\sigma(u)$, where $x\in\R^4$, $u\in\R$, $B\in\R^4$, the $4\times 4$ matrix $A$ is not 
diagolizable and admits a non zero purely imaginary eigenvalue of multiplicity two, the pair $(A,B)$ is 
controllable and $\sigma:\R\to\R$ is a saturation function. We prove that there exists a linear feedback 
$u=K^Tx$ such that the resulting closed loop system given by $\dot x=Ax+B\sigma(K^Tx)$ is globally 
asymptotically stable with respect to the origin.
\end{abstract}

\tableofcontents

\section{Introduction}\label{se:intro}
In this paper, we address the issue of stabilizing a finite dimensional linear control system by means of a saturated control. That is, one has 
\begin{equation}\label{eq:sys}
(\Sigma)\quad \dot x=Ax+B\sigma(u),\ \ x\in \R^n,\ u\in\R^m,
\end{equation}
where $n,m$ are positive integers, and $A$ and $B$ are $n\times n$ and $n\times m$ matrices respectively with real entries. Here $\sigma=(\sigma_i)_{1\leq i\leq m}$,  where each 
$\sigma_i:\R\to\R$ is a saturation function, i.e., any  locally Lipschitz function so that 
there exist positive real numbers $a_1,b_1,a_2,b_2$ where 
$s_{a_1,b_1}(x)\leq \sigma_i(\xi)\leq s_{a_2,b_2}(x)$ for $\xi\in\R$, with 
$s_{a,b}:\R\to\R$ is the function defined for any positive real numbers $a,b$ by 
$s_{a,b}(x)=\frac{a\ x}{\max(b,\vert x\vert)}$. For instance, $\arctan$, $\tanh$ or the standard saturation function $s_{1,1}$ are typical exemples of saturation functions. 
We refer to \cite{Hu-Lin} and \cite{Tarbour} as standard references on the study of these systems in control theory. For the rest of the paper, we will assume that the pair $(A,B)$ is controllable. 

The basic issue consists in finding a continuous feedback law $u=k(x)$ such that the closed system associated with $(\Sigma)$ and $k(\cdot)$ and equal to $\dot x=Ax+B\sigma(k(x))$ is globally asymptotically stable (GAS for short) with respect to the origin. It has been shown (in \cite{SS-90} for instance) that a necessary condition for the existence of such a feedback $k(\cdot)$ is that the real part of any eigenvalue of $A$ is non negative. Note also that optimal control can furnish a stabilizing feedback, which is discontinuous in general. It is not difficult to see that the above mentioned stabilization issue gets not so easy in case where $A$ admits non trivial Jordan blocks associated with purely imaginary eigenvalues. One can first try to seek linear feedbacks, i.e., $k(x)=K^Tx$ with $K$ an $m\times n$ matrix. However, it has been established in \cite{Fuller} that, if $A$ is a Jordan block of order $3$, then 
$(\Sigma)$ cannot be stabilized by a linear feedback law, a result which has been extended  in \cite{SY-91} to the case where $A$ is any Jordan block of order $n\geq 3$. One had therefore to rely on non linear feedback laws $u=k(x)$ and it is in \cite{Teel} that the stabilization issue was solved for Jordan block of order $n\geq 3$ and scalar input (i.e., $m=1$) by using the celebrated feedback referred to as ``nested saturations". Such a feedback has been also used in \cite{SYS-94} to handle the general case described by \eqref{eq:sys}. As a matter of fact, the solution given in that reference relies on a (partial) solution of a more general problem related to $(\Sigma)$, that is its $L_p$-stabilization. Recall that, once a stabilizing feedback $u=k(x)$ has been determined for $(\Sigma)$, one wants to understand its robustness properties and for that purpose, one considers the input-output map $\phi_{k,p}: d\mapsto x_d$, where the disturbance $d$ belongs to $L_p(\R_+,\R^m)$ for some $p\in [1,\infty]$, and $x_d$ is the (unique) solution of $\dot x=Ax+B\sigma(k(x)+d)$ starting at the origin at $t=0$. If $\phi_{k,p}$ takes values in 
$L_p(\R_+,\R^n)$, then the feedback $k(\cdot)$ is said to be $L_p$-stabilizing and it has finite ($L_p$) gain if $\phi_{k,p}$ is a bounded (non linear) operator. In case $A$ is neutrally stable, $(\Sigma)$ is stabilizable by a linear feedback law, which turns out to have finite gain for every $p\in [1,\infty]$, cf. \cite{LCS} while detailed results have been given in \cite{Chitour2001} for the double integrator relatively to $L_p$-stabilization of several feedback laws. In the general case described by \eqref{eq:sys}, the situation is more complicated since the input-output map 
$\phi_{k,p}$ associated with the nested saturation feedback is not $L_p$-stable in general. The first general solution of a 
feedback law for $(\Sigma)$ with finite $L_p$ gain has been given in \cite{Saberi} inspired by a solution given in \cite{Megretsky} for the stabilization issue and based on high and low gain techniques. Note though that the feedbacks provided by \cite{Saberi} are implicit enough to render their use for practical issues rather difficult and therefore a much simpler solution, based on sliding mode ideas, has been provided in \cite{CHL} for a feedback law for $(\Sigma)$ with finite $L_p$ gain for $A$ equal to any Jordan block of order $n\geq 3$ and scalar input. 

One of the issues left open in that long string of research consists in determining conditions for the existence (or non existence) of linear stabilizing feedbacks for $(\Sigma)$ if the state dimension $n$ is larger than two. In particular, the first case not covered by existing results deals with the so called ``complex double integrator" ($CDI$ for short), i.e., one considers $(\Sigma)$ in the special case $n=4$, $m=1$ and $A$ not diagonalizable with two non zero purely imaginary eigenvalues. It means that 
\beq\label{eq:CDI0}
A\hbox{ is similar to }A_\omega:=\begin{pmatrix}\omega A_0&I_2\\0_2&\omega A_0 \end{pmatrix},
\eeq
where $\omega>0$, $I_2$ and $0_2$ are the $2\times 2$ identity and zero matrices respectively and    
$$
A_0=\bpm 0&-1\\1&0\epm.
$$
In the present paper, we bring a positive answer to the stabilization issue associated with 
$CDI$ by means of a linear feedback, in the case where the saturation function is further assumed to be odd, non decreasing and with a derivative non increasing on 
$\R_+$. The main idea consists in embedding $CDI$ into a continuous family of linear control systems with saturated control
$(T_\varepsilon)_{\varepsilon>0}$ so that $CDI=T_1$ and the stabilization by means of a linear feedback of CDI is equivalent to that of $T_\varepsilon$ for any $\varepsilon>0$. Then, in a first step, one characterizes a limit system $T_0$ for $(T_\varepsilon)_{\varepsilon>0}$, as $\varepsilon$ tends to zero, which is GAS with respect to the origin and also a strict Lyapunov function $V$ associated with $T_0$. It is worth noticing that $T_0$ is a linear control system with saturated control with a radial saturation, cf. \eqref{eq:CDI-T0}. 
The second and more complicated step consists in establishing that $T_\varepsilon$ is GAS with respect to the origin, for $\varepsilon$ small enough. This is done by considering $T_\varepsilon$ as a perturbation of $T_0$ and by proceeding at non trivial estimates of the variations of $V$ along trajectories of 
$T_\varepsilon$.

We close this introduction by proposing a conjecture regarding the stabilization issue associated with $(\Sigma)$ by means of a linear feedback under the condition that $(A,B)$ is controllable. We claim that $(\Sigma)$ is stabilizable  by means of a linear feedback if and only if the purely imaginary eigenvalues of $A$ do not admit any Jordan block of order larger than or equal to three.
\section{Notations and statements of the main result}\label{se:not}
If $x\in\R$, let $E(x)$ be its integer part. When $\vp$ tends to $x_0\in\R\cup{\infty}$, the notation $g(\vp)=O(f(\vp))$ means that there exists $C_0>0$ independent of $\vp$ such that $\vert g(\vp)\vert\leq C_0\vert f(\vp)\vert$ as $\vp$ tends to $x_0$ and the notation $g(\vp)=o(f(\vp))$ means that $\vert g(\vp)\vert\leq C(\vp)\vert f(\vp)\vert$ with $C(\vp)>0$ tending to zero as $\vp$ tends to $x_0$.

If $f:\R\to\R$ is a function and $t_1\leq t_2$ two times, we use $\Delta f\Big\vert_{t_1}^{t_2}$ to denote $f(t_2)-f(t_1)$.

For $n,m\in\N^*$, let $M_{n,m}(\mathbb{R})$ (resp. $M_{n,m}(\mathbb{C})$) be the set of $n\times n$ matrix with real (resp. complex) entries and, if $n=m$, we simply use $M_n(\mathbb{R})$ (resp. $M_n(\mathbb{C})$).
We use $(e_1,e_2)$, $I_2\in M_2(\mathbb{R})$ and $J_2\in M_2(\mathbb{R})$ denote the canonical basis of $\R^2$, the identity matrix of $\R^2$ and the $2$-dimensional real \emph{Jordan block}, i.e., $J_2e_i=e_{i-1}$, for $1\leq i\leq 2$ with the convention that $e_0=0$. We also consider $J_2^c\in M_{4}(\mathbb{R})$ the complex Jordan block defined as $J_2^c=J_2\otimes I_2$. 

For $\omega\geq 0$, we define the matrix $J_2(\omega)$ as follows
\beq\label{eq:J2-om}
J_2(0)=J_2,\quad 
J_2(\omega)=\omega I_2\otimes A_0+J_2^c,\hbox 
{ for }\omega>0.
\eeq 
For $\vp>0$, let $D_\vp$ be the $4$-dimensional diagonal matrix defined by 
\beq\label{eq:dvp}
D_\vp=\diag(\vp^2,\vp^2,\vp,\vp).
\eeq
For $\theta\in S^1$, we use $R_\theta$ to denote the rotation of $\R^2$
of angle $\theta$, i.e. the matrix 
$$
R_\theta=\bpm c_\theta&-s_\theta\\s_\theta&c_\theta\epm,
\quad
\hbox{ where }c_\theta:=\cos(\theta),\ s_\theta:=\sin(\theta).
$$
We use $A_0$ to denote $R_{\pi/2}$.

If $x\in\R^2$, we use $x^{\perp}$ to denote $A_0x$, the orthogonal of $x$. If in addition $x\neq 0$, then $x/\Vert x\Vert\in S^1$ and we use $\theta_x\in [0,2\pi)$ the corresponding angle. In particular, $x=\Vert x\Vert R_{\theta_x}e_1=-\Vert x\Vert R_{\theta_x}e_2^{\perp}$.

\begin{definition}[Saturation function]\label{def:sat}
A function $\sigma:\R\to\R$ is called a scalar \emph{saturation} function if it verifies the following:
\begin{description}
\item[$(s1)$] $\sigma$ is an odd and globally Lipschitz function;
\item[$(s2)$] $\sigma(\xi)\xi>0$ for every non zero $\xi\in\R$, and
$$
\lim_{\xi\to +\infty}\sigma(\xi)=\sigma_\infty>0,\quad \lim_{\xi\to 0}\frac{\sigma(\xi)}{\xi}=\sigma'(0)>0;
$$
\item[$(s3)$] $\sigma$ is non decreasing and $\sigma'$ is non increasing on $\mathbb{R}_+$.
\end{description}
\end{definition}
Examples of saturation functions are $\arctan, \tanh$ and the standard saturation function defined by $\sigma_s(\xi)=\frac{\xi}{\max(1,\vert \xi\vert)}$. Note that Item $(s3)$ is not usually considered in the standard definition of saturation function.
\begin{remark}\label{rem:sat}
As easy consequences of the definition, the following holds true:
\begin{description}
\item[$(c1)$] For $\xi\in\R$, consider
\beq\label{eq:Sig} 
\Sigma(\xi)=\int_0^{\xi}\sigma(v)dv.
\eeq
Then $\Sigma$ is an even, positive definite function tending linearly to infinity as $\vert \xi\vert$ tends to infinity;
\item[$(c2)$]  for every non zero $\xi$, one has that 
$\sigma'(\xi)\leq \sigma(\xi)/\xi$ and $\xi\mapsto \sigma(\xi)/\xi$ is an even function, differentiable on $\R^\ast$ and decreasing over $\R_+^\ast$;
\item[$(c3)$] $\sigma'$ is continuous at $\xi=0$ and there exists $\xi_0>0$ such that $\sigma'(\xi)\geq \sigma'(0)/2$ for $\xi\in [-\xi_0,\xi_0]$.
\end{description}
A proof of the above items is given in Appendix.
\end{remark}

\begin{definition}[Stabilizing linear feedback]\label{def:SLF}
Given a linear control system with input subject to saturation $(\Sigma):\ \dot x=Ax+B\sigma(u)$ with 
$x,B\in\R^4$, $u\in\R$, $A\in M_4(\R)$ and 
$\sigma:\R\to\R$ a saturation function. A vector $K\in\R^4$ is called \emph{a 
stabilizing linear feedback for $(\Sigma)$} if the closed loop system $\dot x=Ax+B\sigma(K^Tx)$ is globally 
asymptotically stable (GAS) with respect to the origin.
\end{definition}
In this paper, we prove the following result.
\begin{theorem}\label{th:main}
Let $(CDI)$ be the saturated complex double integrator, that is  the control system 
given by
\beq\label{eq:CDI-0}
(CDI)\quad\dot x=J_2(\omega)x-b\sigma(u),
\eeq
where $x\in\R^4$, $u\in\R$, $\sigma:\R\to\R$ is a saturation function, $\omega>0$ and 
$$
b=\bpm b_1\\b_2\epm,\quad b_i\in\R^2\hbox{ for }i=1,2,
$$
with $(J_2(\omega),b)$ is controllable. Then, $(CDI)$ admits 
a stabilizing linear feedback.
\end{theorem}
\section{Proof of Theorem~\ref{th:main}}\label{se:CDI}
We start the argument by first providing a normal form for $(CDI)$. Since
 $(J_2(\omega),b)$ is controllable, then $b_2$ must be a non zero vector of $\R^2$. Then one gets the 
 following.
\begin{proposition}\label{prop:NF}
The control system $(CDI)$ defined in \eqref{eq:CDI-0} can be brought, up to a linear change of variable and a time rescaling, to the form
\beq\label{eq:CDI-1}
(CDI)_1\quad
\left\{
 \begin{array}{lll}
 \dot x_1&=& 2\pi A_0x_1+x_2,\\  
 \dot x_2&=& 2\pi A_0x_2-e_2\sigma(u),
\end{array}
\right.
\eeq
where $\sigma$ is a saturation function with $\sigma_\infty=\sigma'(0)=1$.
\end{proposition}
\begin{proof}
If $b_1\neq 0$, pick $\alpha>0$ and a rotation $U_1$ so that $b_2=\alpha U_1b_1$.  Perform first the linear 
change of variable given by $(\alpha U_1x_1-x_2,\alpha U_1x_2)$ and then the linear change of variable 
given by $\beta U_2(x_1,x_2)$, with $\beta \alpha U_1U_2b_2=e_2$. One gets that $(\Sigma)$  has been 
brought to the form 
$$
\left\{
 \begin{array}{lll}
 \dot x_1&=& \omega A_0x_1+x_2,\\  
 \dot x_2&=& \omega A_0x_2-e_2\sigma(u).
\end{array}
\right.
$$
Next consider $X_1(t)=\lambda x_1(2\pi t/\omega)$ and $X_2(t)=\lambda x_2(2\pi t/\omega)/\omega$  and $\sigma(k_1u)/k_2$ to conclude for appropriate choices of $\lambda,k_1,k_2>0$.

\end{proof}

One has to determine a stabilizing linear feedback $K\in\R^4$ for 
$(CDI)_1$, i.e., that there exists $K\in\R^4$ such that the closed loop system defined by
\beq\label{eq:CDI-S1}
(S_1)\quad
\left\{
 \begin{array}{lll}
 \dot x_1&=& 2\pi A_0x_1+x_2,\\  
 \dot x_2&=& 2\pi A_0x_2-e_2\sigma(K^Tx),
\end{array}
\right.
\eeq
is GAS with respect to the origin. For that purpose, we imbed $(S_1)$ into a family of dynamical systems $(S_\vp)_{\vp>0}$ defined as follows. For $\vp>0$, the curves $t\mapsto x_\vp(t)=D_\vp x(t/\vp)$, where $t\mapsto x(t)$ is any trajectory of $(S_1)$ and $D_\vp$ has been defined in \eqref{eq:dvp}, are exactly the trajectories of the dynamical system $(S_\vp)$ given by
\beq\label{eq:CDI-Svp}
(S_\vp)\quad
\left\{
 \begin{array}{lll}
 \dot x_1&=& \frac{2\pi A_0}{\vp}x_1+x_2,\\  
 \\
 \dot x_2&=&  \frac{2\pi A_0}{\vp}x_2-e_2\sigma(K_\vp^Tx),
\end{array}
\right.
\quad x=(x_1,x_2)\in\R^4,\quad K_\vp=D_\vp^{-1}K.
\eeq
The following lemma is immediate.
\begin{lemma}\label{le:1-vp}
There exists a stabilizing linear feedback $K_1\in\R^4$ rendering 
$(S_1)$ GAS with respect to the origin if and only if, for every $\vp>0$, there exists a stabilizing linear feedback $K_\vp\in\R^4$ rendering 
$(S_\vp)$ GAS with respect to the origin.
\end{lemma}
The rest of the section is devoted to an argument for
the next proposition.
\begin{proposition}\label{th:GOAL-CI}
There exists $\vp_0>0$ such that, for every $\vp\in (0,\vp_0)$, there exists a stabilizing linear feedback $K_\vp\in\R^4$ rendering 
$(S_\vp)$ GAS with respect to the origin.
\end{proposition}
Proposition~\ref{th:GOAL-CI}, together with Lemma~\ref{le:1-vp}, achieves the stabilisation objective for $(CDI)_1$, i.e., Theorem~\ref{th:main} holds true.

\subsection{Limiting behavior for $(T_\vp)$ as $\vp\to 0$.}\label{sse:lim-beh}
Clearly, understanding the asymptotic behaviour of $(S_\vp)$
for any fixed value of $\vp>0$ is as difficult as fixing $\vp=1$. 
The strategy we follow is made of two steps. In the first one, 
we let $\vp$ tend to zero or infinity and 
expect to characterize a limit system which is GAS with 
respect to the origin. Then, in a second step, considering $
(S_\vp)$ (for $\vp$ small or large enough) as a 
perturbation of the limit system, we aim at extending the 
GAS property of the limit system to neighboring $(S_\vp)$'s.

As $\vp$ tends to infinity, it is not difficult to see that a limit 
system exists (by simply cancelling the terms in $2\pi A_0/\vp$ ), 
but the latter ``contains'' a double integrator and hence it is 
unstable with respect to the origin for any choice of linear 
feedback $K$. In that case, we cannot even complete the first 
step of our strategy. As $\vp$ tends to zero, the term 
$2\pi A_0/\vp$ blows up but the flow associated with this linear 
term corresponds to a rotation and thus remains uniformly 
bounded. Relying on a variation of constant formula, one 
obtains 
a family $(T_\vp)_{\vp>0}$ of dynamical systems on $\R^4$ 
which admits a limit $(T_0)$ as $\vp$ tends to zero in a 
sense precised below.

One passes from $(S_\vp)_{\vp>0}$ to $
(T_\vp)_{\vp>0}$ using the time-varying linear change of 
variable $Y_\vp(t)=R_{-2\pi t/\vp}x(t)$. Setting  
\beq\label{eq:bKvp}
b_\vp(t)=R_{-2\pi t/\vp}e_2,
\eeq
and choosing 
\beq\label{eq:bKvp11}
K_\vp=\bpm e_2\\e_2\epm,
\eeq
an easy computation yields that $Y_\vp=(y_1,y_2)$ is a 
trajectory of 
$$
\left\{
 \begin{array}{lll}
 \dot y_1&=&y_2,\\  
 \dot y_2&=&-b_\vp\sigma(b_\vp^T(y_1+y_2)),
\end{array}
\right.
$$
where we have dropped the time dependence in $b_\vp$ for notational simplicity. 
We finally define $z=y_1+y_2$ and $y=y_2$ to get the following one-parameter family $(T_\vp)_{\vp>0}$ of time-varying dynamical systems on $\R^4$ given by 
\beq\label{eq:CDI-Tvp}
(T_\vp)\quad
\left\{
 \begin{array}{lll}
 \dot z&=& y-b_\vp\sigma(b_\vp^Tz),\\  
 \dot y&=& -b_\vp\sigma(b_\vp^Tz).
\end{array}
\right.
\eeq
It is immediate to see that Proposition~\ref{th:GOAL-CI} holds true if, for $\vp>0$ small 
enough, $(T_\vp)$ is GAS with respect to the origin (with the definition of GAS uniformly 
with respect to time in the case of non autonomous ODEs).

We have the following lemma which is is the key step to identify the limit system $(T_0)$. 
\begin{lemma}\label{le:bvp-conv}
Assume that $\sigma$ is a saturation function as defined in Definition~\ref{def:sat}.
Let $S$ be the modified saturation function associated with $\sigma$ as defined in 
Appendix. Then, the family of time-varying vector fields on $\R^2$, 
$(f_\vp(t,\cdot))_{t\geq 0}$, defined by 
\beq\label{eq:vf2-vp}
f_\vp(t,z)=b_\vp\sigma(b_\vp^Tz),\quad (t,z)\in\R_+\times \R^2,
\eeq
converges, as $\vp$ tends to zero, to the vector field $f:\R^2\to\R^2$ given by 
\beq\label{eq:f}
f(z)=\left\{
 \begin{array}{lll}
S(\Vert z\Vert)\frac{z}{\Vert z\Vert}&\hbox{ if }&z\neq 0,\\  
 0&\hbox{ if }&z=0,
\end{array}
\right.
\eeq
for the weak-$\ast$ topology of $L^{\infty}(\R_+,\R^2)$, i.e., for every $z\in\R^2$ and $g\in L^{1}(\R_+,\R^2)$, 
$$
\lim_{\vp\to 0}\int_0^{\infty}f^T_\vp(t,z)g(t)dt=f^T(z)\int_0^{\infty}g(t)dt,
$$ 
and the above convergence is uniform with respect to $z\in\R^2$. 
\end{lemma}
\begin{proof} It is enough to show that for every $0\leq a<c$, one has $\lim_{\vp\to 0}I_\vp=f(z)$, where 
\beq\label{eq:cv-w1}
I_\vp=\frac1{c-a}\int_a^cf_\vp(t,z)dt,
\eeq
and that the convergence is uniform with respect to $z\in\R^2$. For $z=0$, the result is true 
with no limit involved. 

Hence we suppose in the sequel that $z\neq 0$. Since $z=-\Vert z\Vert R_{\theta_z}e_2^{\perp}$, one has that 
$$
b_\vp^Tz=-\Vert z\Vert e_2^TR_{-2\pi t/\vp-\pi/2}R_{\theta_z}R_{\pi/2}e_2=\Vert z\Vert s_{\theta_z+2\pi t/\vp}.
$$
Hence one has that
$$
I_\vp=\frac1{c-a}\int_a^c\sigma(\Vert z\Vert s_{\theta_z+2\pi t/\vp})b_\vp dt.
$$
After performing the change of time $v=\theta_z+2\pi t/\vp$, one gets that $
I_\vp=\frac{1}{c-a}R_{\theta_z}J_\vp$ where 
$$
J_\vp=\frac{\vp}{2\pi }\int_{\theta_z+2\pi a/\vp}^{\theta_z+2\pi c/\vp}
\sigma(\Vert z\Vert s_v)\bpm s_v\\c_v\epm dv.
$$
Set $k=E(\frac{2\pi(c-a)}{\vp})$. Then
$$
J_\vp=O(\vp)+\frac{\vp}{2\pi } \int_{\theta_z+2\pi a/\vp}^{\theta_z+2\pi a/\vp+k}
\sigma(\Vert z\Vert s_v)\bpm s_v\\c_v\epm dv=O(\vp)+k\vp\int_0^{1}\sigma(\Vert z\Vert s_{2\pi v})\bpm s_{2\pi v}\\c_{2\pi v}\epm dv,
$$
where the last equality holds since $v\to 
\sigma(\Vert z\Vert s_{2\pi v})\bpm s_{2\pi v}\\c_{2\pi v}\epm$ is $1$-periodic. Moreover the terms $O(\vp)$ do not depend on $z$. It is 
then immediate to compute that
$$
\int_0^{1}\sigma(\Vert z\Vert s_{2\pi v})\bpm s_{2\pi v}\\c_{2\pi v}\epm dv=S(\Vert z\Vert)\bpm 1\\0\epm.
$$
Since $R_{\theta_z}\bpm 1\\0\epm=\frac{z}{\Vert z\Vert}$, the lemma is proved.

\end{proof}
According to the previous lemma, the one-parameter family of time-varying dynamical systems $(T_\vp)_{\vp>0}$ converges for the weak-$\ast$ topology of $L^{\infty}(\R_+,\R^4)$ to the dynamical system $(T_0)$ defined on $\R^4$ by 
\beq\label{eq:CDI-T0}
(T_0)\quad
\left\{
 \begin{array}{lll}
 \dot z&=& y-f(z),\\
 \dot y&=& -f(z),
 \end{array}
\right.
\eeq
where the vector field $f$ on $\R^2$ has been defined in 
\eqref{eq:f}. 
To study $(T_0)$, we need the following lemma.
\begin{lemma}\label{le:gradient-f}
Let $f:\R^2\to\R^2$ be the vector field defined in \eqref{eq:f}.
Then $f$ is bounded, of class $C^1$ and, for every $(z,y)\in\R^4$, one has
\beq\label{eq:gradient-f}
y^T\Big(f(z+y)-f(z)\Big)\geq 0,
\eeq
with equality if and only if $y=0$.
\end{lemma}
\begin{proof}
From Proposition~\ref{prop:mod-sat}, we have that $f$ is bounded and,  since $S$ is of class $C^1$ and $\xi\to S(\xi)/\xi$ is decreasing, $f$ is differentiable everywhere, $C^1$ outside the origin and $df(0)=S'(0)I_2$. Indeed, for $z\neq 0$, one has that
\beq\label{eq:gradient-f-1}
df(z)=S'(\Vert z\Vert)\frac{zz^T}{\Vert z\Vert^2}+\frac{S(\Vert z\Vert)}{\Vert z\Vert}
\Big(I_2-\frac{zz^T}{\Vert z\Vert^2}\Big). 
\eeq
Note that, since $z\in\R^2$, one has that  $I_2-\frac{zz^T}{\Vert z\Vert^2}=\frac{z^{\perp}(z^{\perp})^T}{\Vert z\Vert^2}$. Clearly $df(z)$ is bounded and continuous at $z=0$. Moreover, since both $S'$ and $\xi\mapsto S(\xi)/\xi$  are positive functions, then $df(z)$ is symmetric positive definite for every $z\in\R^2$. 

For every $(z,y)\in\R^4$, one has 
\beq\label{eq:S'S}
y^T\Big(f(z+y)-f(z)\Big)=\int_0^1y^Td f(z+sy)y\,ds,
\eeq
which is clearly non negative, and strictly positive if $y\neq 0$ since $z\mapsto df(z)$ is everywhere positive definite.
\end{proof}
As a consequence of Lemma~\ref{le:gradient-f}, we have the following proposition, which describes the asymptotic behaviour of trajectories of $(T_0)$.
\begin{proposition}\label{prop:asymp-T0}
Trajectories of $(T_0)$ given in \eqref{eq:CDI-T0} are defined for all non negative times.
Moreover, consider the function $V_0:\R^4\to\R_+$ given by
\beq\label{eq:V0}
V_0(z,y)=\Vert y\Vert^2+\int_0^{\Vert z\Vert}S(\xi)d\xi+\int_0^{\Vert z-y\Vert}S(\xi)d\xi.
\eeq
Then $V_0$ is a $C^1$, positive definite and radially unbounded function which is a strict Lyapunov function along trajectories of $(T_0)$. As a consequence, $(T_0)$ is GAS with respect to the origin.
\end{proposition}
\begin{proof} The vector field on $\R^4$ defining $(T_0)$ is $C^1$, thanks to Lemma~\ref{le:gradient-f}, and, since its growth at infinity is linear, trajectories of $(T_0)$ are defined for all non negative times. Properties of $V_0$ are immediate and we next check that $V_0$ is a strict Lyapunov function for $(T_0)$.
Indeed, if we use $\dot V_0$ to denote the time derivative of $V_0$  along non trivial trajectories of $(T_0)$, one gets that
\begin{eqnarray}
\dot V_0&=&-S(\Vert z\Vert)^2-y^T\Big(f(z)-f(z-y)\Big)=-S(\Vert z\Vert)^2-\int_0^1y^Td f(g(s))y\,ds\nonumber\\
&=&-S(\Vert y\Vert)^2-\int_0^1\Big[S'(\Vert g(s)\Vert)\frac{(y^Tg(s))^2}{\Vert g(s)\Vert^2}+\frac{S(\Vert g(s)\Vert)}{\Vert g(s)\Vert}\frac{(y^Tg(s)^{\perp})^2}{\Vert g(s)\Vert^2}\Big]ds,
\label{eq:dV_0}
\end{eqnarray}
where $g(s)=z-(1-s)y$ for $s\in[0,1]$. One gets the conclusion by using Lemma~\ref{le:gradient-f}.

\end{proof}
\begin{remark}
Note that $(T_0)$ is locally exponentially stable at the origin since the linearized system associated with $(T_0)$ at the origin is defined by the Hurwitz matrix
$$
\bpm -S'(0)I_2&I_2\\-S'(0)I_2&0\epm=\bpm -S'(0)&1\\-S'(0)&0\epm\otimes I_2.
$$

\end{remark}
\begin{remark} Recall that the double integrator (DI) is the linear control system defined on $\R^2$ by
$\dot x=J_2x+e_2u$. For any feedback $u=-\sigma(k^Tx)$ where $k\in\R^2$ has positive coordinates and $\sigma$ is a saturation function, the closed loop system  $\dot x=J_2x-e_2\sigma(k^Tx)$ is GAS with respect to the origin. After a linear change of variable and time, such a system can be brought to the form corresponding to $(T_0)$ namely
\beq\label{eq:DI}
(DI)\quad
\left\{
 \begin{array}{lll}
 \dot z&=& y-\sigma(z),\\  
 \dot y&=& -\sigma(z),
\end{array}
\right.
\eeq
with $(z,y)\in\R^2$. It has been proved in \cite{These-YY} that the radially unbounded positive definite function $V:
\R^2\to\R_+$ given by
$$
V(z,y)=y^2+\int_0^{z}\sigma(\xi)d\xi+\int_0^{z-y}\sigma(\xi)d\xi,
$$
is a strict Lyapunov function for $(DI)$.
It is immediate to see that $V_0$ is a simple adaptation of $V$ to $(T_0)$.
\end{remark}
\begin{remark}
Let $F_2:\R^4\to\R^4$ be the vector field on $\R^4$ defining $
(T_0)$. It is rather immediate to see that, for every $n\geq 1$, 
one can define a vector field $F_n$ on $\R^{2n}$ where 
$F_n(z,y)$ is defined exactly as $F_2(z,y)$, now with $z$ and $y$ 
vectors in $\R^n$. (For $n=1$, $z/\Vert z\Vert$ must be 
understood as the sign of $z\in\R$.) Then the conclusions of 
Proposition~\ref{prop:asymp-T0} extend verbatim to $F_n$ 
with the same Lyapunov function $V_0$ now defined on $
\R^{2n}$.
\end{remark}
\subsection{Study of $(T_\vp)$ for $\vp$ small enough.}\label{ssection-vp-small}
By characterizing $(T_0)$, we have achieved the first step of the strategy devised to prove Proposition~\ref{th:GOAL-CI}.
We next turn to the second step and for that purpose we will analyse the variations of $V_0$ along trajectories of $(T_\vp)$ for $\vp$ small enough.

The time derivative $\dot V_0$ of $V_0$  along non trivial trajectories of $(T_\vp)$ is given by 
\begin{eqnarray}
\dot V_0&=&-S(\Vert z\Vert)\Big(b_\vp^T\frac{z}{\Vert z\Vert}\Big)
\sigma(b_\vp^Tz)
\label{eq:t1}\\
&-&y^T\Big(f(z)-f(z-y)\Big)\label{eq:t2}\\
&+&2y^T\Big(f(z)-b_\vp^T\sigma(b_\vp^Tz)\Big)\label{eq:t3}.
\end{eqnarray}
Clearly the two first terms \eqref{eq:t1} and \eqref{eq:t2} are non positive and one must handle the effect of 
the third one \eqref{eq:t3}. As a matter of fact, if $b_\vp^Tz=0$ and $z=y\neq 0$, then $\dot V_0=S(\Vert 
z\Vert)\Vert z\Vert$, which is positive and unbounded over $\R^2$. Then $V_0$ cannot be a Lyapunov 
function for $(T_\vp)$ for $\vp>0$ since clearly $\dot V_0$ (the time derivative of $V_0$  along non trivial 
trajectories of $(T_\vp)$) can clearly be positive. To circumvent this problem, we will evaluate variations of 
$V_0$ on appropriate time intervals when $\Vert (z,y)\Vert$ is large.

\begin{remark}
One could have also written $\dot V_0$ as
$$
\dot V_0=-S(\Vert z\Vert)^2-y^T\Big(f(z)-f(z-y)\Big)+(2y-f(z))^T\Big(f(z)-b_\vp^T\sigma(b_\vp^Tz)\Big),
$$
with a more handleable first term since it is $\vp$-free. However, it introduces an extra quantity in the third term, which turns out to be not so easy to deal with.
\end{remark}

We aim at establishing the following key technical proposition.
\begin{proposition}\label{prop:key1}
There exists $\vp_0>0$, $R,C_1>0$ and $\rho\in (0,1)$, such that, for every $\vp\in (0,\vp_0)$, $(z_0,y_0)\in\R^4$ with $V(z_0,y_0)\geq R$, there exists $T(z_0,y_0)$
such that 
\beq\label{eq:T-CDI}
\rho\max(1,\Vert y_0\Vert)\leq T(z_0,y_0)
\leq 2\rho\max(1,\Vert y_0\Vert),
\eeq
for which
\beq\label{eq:key2}
\Delta V_0\Big\vert_0^{T(z_0,y_0)}\leq -C_1T(z_0,y_0),
\eeq
along every trajectory of $(T_\vp)$ starting at $(z_0,y_0)$.
\end{proposition} 
\begin{proof}
The several constants will be fixed along the argument but typically $\vp_0$ and $\rho$
will be small compared to one while $R$ will be large compared to one. Let us stress that $\rho$, $R$ and 
$\vp_0$ will be eventually modified in the argument (typically by decreasing $\rho$ and $\vp_0$ and 
increasing $R$) but these choices remain ``universal'', i.e., only depending on $\vp_0$ and thus independent 
of $\vp<\vp_0$. We will also use the symbol $C_R$ to denote positive constants that only depend on $R$ 
and $\sigma$.

We fix $(z_0,y_0)\in\R^4$ with $V(z_0,y_0)\geq R$ and simply use $T$ to denote $T(z_0,y_0)$. 
Note that $\dot V_0\geq -1-3\sqrt{V_0}$. In particular, as long 
as $V_0\geq 1$, one has that $ (\sqrt{V_0}\dot) \geq -2$ and 
hence $\sqrt{V_0}\geq \sqrt{V_0(0)}(1-4\rho\pi)$ on $[0,T\pi]$. 
In particular, $\sqrt{V_0}\geq R^{1/2}/2$ on $[0,T\pi]$. 

There are two key quantities to estimate, namely
\beq\label{eq:Lvp}
L_\vp=-\int_0^{T}S(\Vert z\Vert)\Big(b_\vp^T\frac{z}{\Vert z\Vert}\Big)\sigma(b_\vp^Tz)dt,
\eeq
and 
\beq\label{eq:Kvp}
K_\vp=K^1_\vp+K^2_\vp,
\eeq
where
\beq\label{eq:Kvp-1}
K^1_\vp=-\int_0^{T}y^T\Big(f(z)-f(z-y)\Big)dt,\quad 
K^2_\vp=2\int_0^{T}y^T\Big(f(z)-b_\vp^T\sigma(b_\vp^Tz)\Big)dt.
\eeq
Assume that
\beq\label{eq:Lvp-1}
L_\vp\leq -3C_1T,
\eeq
and 
\beq\label{eq:Kvp-1}
K_\vp\leq (C_1+C_2\vp) T,
\eeq
for some positive constants $C_1,C_2$ independent of $\vp$ small enough. Clearly the above two inequalities yield \eqref{eq:key2}.

\end{proof}
We are now left to establish \eqref{eq:Lvp-1} and \eqref{eq:Kvp-1}.
This is the purpose of the next two lemmas.
\begin{lemma}\label{le:L}
With the above notations, there exists a positive constant such that 
\eqref{eq:Lvp-1} holds true.
\end{lemma}
\begin{proof}
We distinguish two cases.
\begin{description}
\item[$(L1)$] For every $t\in[0,T]$, one has $\Vert y(t)\Vert\leq \frac{\pi \Vert z(t)\Vert}{\vp}$;
\item[$(L2)$] there exists $\bar{t}\in[0,T]$ such that $\Vert y(\bar{t})\Vert>\frac{\pi\Vert z(\bar{t})\Vert}{\vp}$.
\end{description}
Assume that $(L1)$ holds true. Then $z(t)\neq 0$ for every $t\in[0,T]$ and $\theta_z(t)$ is well defined and absolutely continuous. Moreover
$$
\dot\theta_z=\frac{(z^\perp)}{\Vert z\Vert}\frac{d}{dt}{\Big(\frac{z}{\Vert z\Vert}\Big)}=\frac{(z^\perp)^T\dot z}{\Vert z\Vert^2}.
$$
Taking into account the estimate in $(L1)$, one gets that 
$\vert\dot\theta_z\vert \leq 4\pi/3\vp$ on $[0,T]$. 

In the case where $\Vert y_0\Vert\leq R^{1/2}/3$, then $\Vert y\Vert \leq R^{1/2}/2$ and $\Vert z\Vert\geq R/2>1$ on $[0,T]$ for $R$ universal constant large enough.
Assume now that $\Vert y_0\Vert>R^{1/2}/3$. It is immediate to see that 
\beq\label{eq:est-y1}
(1-2\rho)\Vert y_0\Vert \leq \Vert y(t)\Vert \leq (1+2\rho)\Vert y_0\Vert,\quad t\in [0,T].
\eeq
 On the other hand,
let 
\beq\label{eq:Ez1}
E_z:=\{t\in[0,T] \mid \Vert z(t)\Vert< 1\}.
\eeq
If $E_z$ is not empty, let $\tilde{t}\in E_z$. From the dynamics 
and \eqref{eq:est-y1}, one gets that 
$$
z(t)=z(\tilde{t})+(t-\tilde{t})\Big(y(\tilde{t})+O(1)\Big),\ t\in[0,T]
$$
where $\Vert O(1)\Vert$ can be chosen smaller than one, thanks 
to Proposition~\ref{prop:NF}. This implies that
\beq\label{eq:est-y2}
\Vert z(t)\Vert\geq \Big(\vert t-\tilde{t}\vert-\vp\Big)(\Vert 
y(\tilde{t})\Vert-1).
\eeq
From that, it is easy to deduce that $E_z$ is contained in an subinterval of $[0,T\pi]$ of length smaller than 
$2/\Vert y_0\Vert$ and hence there exists a subinterval $I_L$ 
of $[0,T]$ of length at least $T/2$ such that 
for $t\in I_L$, 
\begin{itemize}
\item $\Vert z(t)\Vert\geq 1$,
\item $\vert\dot\theta_z\vert \leq 4\pi/3\vp$.
\end{itemize}
Then one gets, 
\beq\label{eq:Lest1}
L_\vp\leq -\int_{I_L}S(\Vert z\Vert)\Big(b_\vp^T\frac{z}{\Vert z\Vert}\Big)
\sigma(b_\vp^Tz)dt\leq-S(1)\int_{I_L}s_{2\pi t/\vp+\theta_z}
\sigma(s_{2\pi t/\vp+\theta_z})dt,
\eeq
since both $S$ and $\sigma$ are increasing. We now perform the change of time $\tau(t)=2\pi t/\vp+\theta_z$. Since 
$2\pi/3\vp\leq \dot\tau\leq 5\pi/3\vp$ on $I_L$, $t\mapsto\tau(t)
$ realises an increasing bijection between $I_L$ and an 
interval $\tilde{I}_L$ with 
$2\pi\vert I_l\vert/3\vp\leq \vert \tilde{I}_l\vert\leq 4\pi\vert I_l\vert\vp$. 
One deduces from \eqref{eq:Lest1} the following
\beq\label{eq:Lest2}
L_\vp\leq -\frac{3\vp S(1)}{10\pi}\int_{\tilde{I}_L}s_\tau\sigma(s_\tau)d\tau.
\eeq
Since $\tau\mapsto s_\tau\sigma(s_\tau)$ is $\pi$-periodic, it is easy to see that $\int_{\tilde{I}_L}s_\tau\sigma(s_\tau)d\tau\geq S(1)T\pi/6\vp+O(1)$, which implies that $L_\vp\leq-4C_1T+T\vp O(1)$ for some universal constant $C_1$. Then \eqref{eq:Lvp-1} holds if $(L1)$ holds true.

We now assume that $(L2)$ holds true. In particular we have 
that $\Vert y(\bar{t})\Vert\geq R^{1/2}/2$ and both 
\eqref{eq:est-y1} and \eqref{eq:est-y2} hold true.
It is immediate to see that, outside an interval $I_{bad}\subset 
[0,T]$ of length at most $4\pi\vp$ and containing $\bar{t}$, one 
has $\Vert y(t)\Vert\leq \frac{\pi \Vert z(t)\Vert}{\vp}$. We can 
therefore select a subinterval of $[0,T]$ of length at least 
$T/2$ on which the previous inequality holds true on it. We 
are back to $(L1)$ and that concludes the proof of 
\eqref{eq:Lvp-1}.
\end{proof}

\begin{lemma}\label{le:K}
With the above notations, \eqref{eq:Kvp-1} holds true.
\end{lemma}
\begin{proof}
In the sequel, we will use the notation $O(\cdot)$ only when the involved bounds do not depend on $\vp$.
We first perform the change of time $s=t/\vp$ and rewrite $K^1_\vp$, $K^2_\vp$ defined in \eqref{eq:Kvp-1} as 
\beq\label{eq:K1vp}
K^1_\vp=
-\vp\int_0^{T/\vp}y(\vp s)^T\Big(f(z(\vp s))-f(z(\vp s)-y(\vp s))\Big)ds
\eeq
and
\beq\label{eq:K2vp}
K^2_\vp=2\vp\int_0^{T/\vp}y(\vp s)^T\Big(f(z(\vp s))-b_\vp^T\frac{z(\vp s)}{\Vert z(\vp s)\Vert}\Big)ds.
\eeq
We start by several trivial remarks. With our choice of $T$ and since $\dot y=O(1)$, then
 clearly $\Vert y\Vert=O(\max(1,\Vert y_0\Vert))$.
We can therefore always assume that $T/\vp$ is an integer since otherwise the error made in \eqref{eq:K2vp} is $\vp O(\max(1,\Vert y_0\Vert))=\vp O(T)$ and hence negligible if we establish \eqref{eq:Lvp-1}. 

We can then set $T/\vp=k$. We now decompose the integral terms in $K^1_\vp$ and $K^2_\vp$ according to
$$
\int_0^{k}\cdots=\sum_{j=0}^{k-1}\int_{j}^{j+1}\cdots,
$$
and then perform the change of times $s=j+v$ in each interval $[j,(j+1)]$. We deduce from \eqref{eq:K1vp} and \eqref{eq:K2vp} that, for $0\leq j\leq k-1$, one has 
\beq\label{eq:Kvp-dec}
K^1_\vp=-\vp\sum_{j=0}^{k-1}K^1_{\vp,j},\quad
K^2_\vp=2\vp\sum_{j=0}^{k-1}K^2_{\vp,j},
\eeq
where
\beq\label{eq:Kvp-2}
K^1_{\vp,j}=\int_0^{1}y^T\Big(f(z)-f(z-y)\Big)dv,\quad
K^2_{\vp,j}=\int_0^{1}y^T\Big(f(z)-b_1\sigma(b_1^Tz)\Big) dv,
\eeq
where the argument of both $z,y$ is equal to $j\pi\vp+\pi\vp v$.

We need the following notations,
\beq\label{eq:notations0}
z_j=z(\vp j),\ y_j=y(\vp j),\ z_j(v)=z_j+\vp vy_j,\ 0\leq j\leq k-1,\ v\in[0,1].
\eeq
We also have the following estimates, easily deduced from \eqref{eq:notations0},
\beq\label{eq:notations1}
y(\vp j+\vp v)=y_j+\vp O(1)v,\
z(\vp j+\vp v)=z_j(v)+\vp O(1)v,
\eeq
where $\Vert O(1)\Vert\leq 1$.

We next consider, for $0\leq j\leq k-1$, the quantities 
$\widetilde{K}^1_{\vp}$ and $\widetilde{K}^2_{\vp}$ obtained 
as $K^1_{\vp}$ and $K^2_{\vp}$ in \eqref{eq:Kvp-dec} but, instead of $K^1_{\vp,j}$ and $K^2_{\vp,j}$, we use the
integrals 
\beq\label{eq:Kvp-2}
\widetilde{K}^1_{\vp,j}= \int_0^{1}y_j^T\Big(f(z)-f(z-y)\Big)dv,\quad
\widetilde{K}^2_{\vp,j}= \int_0^{1}y_j^T\Big(f(z)-b_{1}\sigma(b_{1}^Tz)\Big) dv,
\eeq
where still the argument of $z$ is equal to $j\pi\vp+\pi v\vp$. From 
\eqref{eq:notations1} and the fact that $f$ and $\sigma$ are bounded, one gets that, for $0\leq j\leq k-1$,
\beq\label{eq:errorKj1}
K^1_{\vp,j}=\widetilde{K}^1_{\vp,j}+O(1)\vp,\quad
K^2_{\vp,j}=\widetilde{K}^2_{\vp,j}+O(1)\vp.
\eeq
One deduces that, for $i=1,2$, 
\beq\label{eq:errorKj2}
K^i_{\vp}=\widetilde{K}^i_{\vp}+\vp^2kO(1)=\widetilde{K}^i_{\vp}+\vp T\pi O(1).
\eeq
Setting $\widetilde{K}_{\vp}=\widetilde{K}^1_{\vp}+\widetilde{K}^2_{\vp}$, one deduces from the previous equation that the argument amounts to prove the estimate \eqref{eq:Kvp-1} for $\widetilde{K}_{\vp}$.

We claim that, for $0\leq j\leq k-1$, one has that
\beq\label{eq:b1-j}
-\int_0^{1}y_j^Tb_{1}\Big(\sigma(b_{1}^Tz)-\sigma(b_{1}^Tz_j)\Big)dv
\leq O(1)\vp.
\eeq
Observe first that one gets from \eqref{eq:notations1}
\beq\label{eq:obs0}
b_{1}^Tz=b_{1}^Tz_j+\vp v\Big(b_{1}^Ty_j+O(1)\Big).
\eeq
To get the claim, one can see that
$$
y_j^Tb_{1}\Big(\sigma(b_{1}^Tz)-\sigma(b_{1}^Tz_j)\geq 0,
$$
as soon as
$b_{1}^Tz_j(b_{1}^Tz-b_{1}^Tz_j)>0$ since $\sigma$ is increasing.
By \eqref{eq:obs0}, the previous inequality does not hold true only if
$\vert y_j^Tb_{2\pi v}\vert =O(1)$, in which case,
$$
\vert y_j^Tb_{1}\Big(\sigma(b_{1}^Tz)-\sigma(b_{1}^Tz_j)\Big)\vert=\vp O(1).
$$
This concludes the argument of the claim \eqref{eq:b1-j}.

Noticing that
$$
\int_0^{1}y_j^Tb_{1}\sigma(b_{1}^Tz_j)dv=y_j^Tf(z_j),
$$
and using \eqref{eq:b1-j}, one deduces that in the estimate $\widetilde{K}^2_{\vp,j}$, one can replace $\sigma(b_{1}^Tz)$ by $f(z_j)$. We are therefore left to show that the following quantity 
\beq\label{eq:Mvp-0}
\vp\sum_{j=0}^{k-1}\int_0^{1}y_j^T\Big(f(z)+f(z-y)-2f(z_j)\Big)dv
\eeq
satisfies the estimate \eqref{eq:Kvp-1}.

Notice that, for $0\leq j\leq k-1$,
$$
f(z)+f(z-y)-2f(z_j)=f(z)-f(z_j)+f(z-y)-f(z_j)=\vp \Big(y_j+O(1)\Big).
$$
One deduces that if $\Vert y_0\Vert\leq R^{1/2}$, then $\Vert y_j\Vert=O(1)$ and $\widetilde{K}^i_{\vp,j}\leq \vp O(1)$ for $i=1,2$ and $0\leq j\leq k-1$, which yields the desired estimate for $K_\vp=\vp TO(R)$. 

We can hence assume that  $\Vert y_0\Vert\geq R^{1/2}$ and then, 
 $\Vert y_0\Vert(1-\rho)\leq \Vert y_j\Vert\leq \Vert y_0\Vert(1+\rho)$ for $0\leq j\leq k-1$.

Similarly to \eqref{eq:b1-j}, we claim that, for $0\leq j\leq k-1$,
one has 
\beq\label{eq:j+1}
\int_0^{1}y_j^T\Big(f(z)+f(z-y)-f(z_j(1))-f(z_j(1)-y_j)\Big)dv\leq O(1)\vp.
\eeq
Indeed, we get from \eqref{eq:notations1} that
$$
z=z_j(1)-\vp \Big((1-v)y_j+vO(1)\Big),\quad
z-y=z_j(1)-y_j-\vp\Big((1-v)y_j+vO(1)\Big).
$$
For $0\leq j\leq k-1$ and, as long as $(1-v)\Vert y_j\Vert>vO(1)$,
one deduces from \eqref{eq:gradient-f} that 
$$
y_j^T\Big(f(z)-f(z_j(1))\Big)\leq 0,\
y_j^T\Big(f(z-y)-f(z_j(1)-y_j)\Big)\leq 0.
$$
The inequality $(1-v)\Vert y_j\Vert\leq vO(1)$ occurs for $v$ close to $1$ and on a subinterval of length $O(1)/\Vert y_j\Vert$.
Using on that subinterval that $f$ is globally Lipschitz, one derives 
\eqref{eq:j+1}.

From \eqref{eq:Mvp-0} and \eqref{eq:j+1}, the argument of Lemma~\ref{le:K} reduces to prove that the quantity $M_\vp$ defined by 
\beq\label{eq:Mvp}
M_\vp=\vp\sum_{j=0}^{k-1}M_{\vp,j},\quad
M_{\vp,j}=y_j^T\Big(f(z_j(1))+f(z_j(1)-y_j)-2f(z_j)\Big),\ 0\leq j\leq k-1,
\eeq
satisfies the estimate \eqref{eq:Kvp-1}. 

For $0\leq j\leq k-1$, set $x_j(v)=z_j-vy_j$ for $v\in[0,1]$. Notice that 
$$
z_j=z_j(0)=x_j(0),\  z_j-y_j=z_j(0)-y_j=x_j(1),
$$
and then one can rewrites \eqref{eq:Mvp} as 
\beq\label{eq:Mj00}
M_{\vp,j}=y_j^T\Big(f(z_j(1))-f(z_j(0))+f(z_j(1)-y_j)-f(z_j(0)-y_j)+f(x_j(1)-f(x_j(0))\Big).
\eeq
By using \eqref{eq:S'S} in the previous equality, one has for every $0\leq j\leq k-1$ that
\beq\label{eq:Mj2}
M_{\vp,j}=\int_0^1\Big(M_{\vp,j}^1(v)+M_{\vp,j}^2(v)\Big)dv,
\eeq
where 
\begin{eqnarray}
M_{\vp,j}^1(v)&=&\vp\left[S'(\Vert z_j(v)\Vert)\frac{(y_j^Tz_j(v))^2}{\Vert z_j(v)\Vert^2}
+
S'(\Vert z_j(v)-y_j\Vert)\frac{(y_j^T(z_j(v)-y_j)^2}{\Vert z_j(v)-y_j\Vert^2}\right]\nonumber\\
&-&
S'(\Vert x_j(v)\Vert)\frac{(y_j^Tx_j(v))^2}{\Vert x_j(v)\Vert^2}\label{eq:Mj1}
\end{eqnarray}
and
\begin{eqnarray}
M_{\vp,j}^2(v)&=&\vp\left[\frac{S(\Vert z_j(v)\Vert)}{\Vert z_j(v)\Vert}\frac{(y_j^Tz_j(v)^{\perp})^2}{\Vert z_j(v)\Vert^2}+
\frac{S(\Vert z_j(v)-y_j\Vert)}{\Vert z_j(v)-y_j\Vert}\frac{(y_j^T(z_j(v)-y_j)^{\perp})^2}{\Vert z_j(v)-y_j\Vert^2}\right]\nonumber\\
&-&\frac{S(\Vert x_j(v)\Vert)}{\Vert x_j(v)\Vert}\frac{(y_j^Tx_j(v)^{\perp})^2}{\Vert x_j(v)\Vert^2}.\label{eq:Mj2}
\end{eqnarray}
Moreover note that, for every $0\leq j\leq k-1$ and $v\in[0,1]$, one has
\beq\label{eq:perp1}
y_j^Tz_j(v)^{\perp}=y_j^T(z_j(v)-y_j)^{\perp}=y_j^Tx_j(v)^{\perp}=y_j^Tz_j.
\eeq
To obtain the required estimate, we subdivide the discussion into two cases and consider a constant $C_*$ large with respect to one, which will be fixed later.

\underline{Case $1$.} For every $t\in [0,T]$, one has that $\Vert z-y\Vert\leq C_*\Vert z\Vert/2$. We will prove that $M_{\vp,j}^1(v)+M_{\vp,j}^2(v)<0$ for every $0\leq j\leq k-1$ and $v\in[0,1]$.

As a consequence of the case assumption, one gets, for every $0\leq j\leq k-1$ and $v\in[0,1]$ that
\beq\label{eq:xyz}
\Vert x_j(v)\Vert\leq C_*\Vert z_j(v)\Vert,\ \Vert x_j(v)\Vert\leq C_*\Vert z_j(v)-y_j\Vert.
\eeq
Using Item  $(S2)$ in Proposition~\ref{prop:mod-sat}, one has, for every $0\leq j\leq k-1$ and $v\in[0,1]$, that
\beq\label{eq:S/s}
\frac{S(\Vert z_j(v)\Vert)}{\Vert z_j(v)\Vert}\leq C_*\frac{S(\Vert x_j(v)\Vert)}{\Vert x_j(v)\Vert},\quad
\frac{S(\Vert z_j(v)-y_j\Vert)}{\Vert z_j(v)-y_j\Vert}\leq C_*\frac{S(\Vert x_j(v)\Vert)}{\Vert x_j(v)\Vert}.
\eeq
By taking into account \eqref{eq:perp1}, one has that 
$$
\frac{(y_j^TZ)^2}{\Vert y_j\Vert^2{\Vert Z\Vert^2}}\leq C_*^2\frac{(y_j^Tx_j(v)^{\perp})^2}{\Vert y_j\Vert^2{\Vert x_j(v)\Vert^2}},
$$
where 
$$
Z\in\{z_j(v)^{\perp},(z_j(v)-y_j)^{\perp}\}.
$$
Then, one deduces from the previous inequalities and \eqref{eq:S/s} that, for every $0\leq j\leq k-1$ and $v\in[0,1]$, one has that
\beq\label{eq:estM2}
M_{\vp,j}^2(v)\leq (\vp C_*^3-1)\frac{S(\Vert x_j(v)\Vert)}{\Vert x_j(v)\Vert}\frac{(y_j^Tx_j(v)^{\perp})^2}{\Vert x_j(v)\Vert^2}\leq 0,
\eeq
where the last inequality is obtained for $\vp$ small enough.

To handle $M_{\vp,j}^1(v)$, first notice that, for every $0\leq j\leq k-1$ and $v\in[0,1]$, one can deduce from the case assumption and Item  $(S3)$ in Proposition~\ref{prop:mod-sat} that
\beq\label{eq:S'1}
S'(\Vert z_j(v)\Vert)\leq \frac{C_*^3}{C_0}S'(\Vert x_j(v)\Vert),\quad
S'(\Vert z_j(v)-y_j\Vert)\leq \frac{C_*^3}{C_0}S'(\Vert x_j(v)\Vert).
\eeq
In the case where 
\beq\label{eq:cos-jv1}
\frac{(y_j^Tx_j(v))^2}{\Vert y_j\Vert^2{\Vert x_j(v)\Vert^2}}\geq 1/\sqrt{2},
\eeq
one deduces that 
\beq\label{eq:estM1-2}
M_{\vp,j}^1(v)\leq \Big(\frac{4\vp C_*^3}{C_0}-1\Big)\frac{S(\Vert x_j(v)\Vert)}{\Vert x_j(v)\Vert}\frac{(y_j^Tx_j(v)^{\perp})^2}{\Vert x_j(v)\Vert^2}\leq 0,
\eeq
where the last inequality is obtained for $\vp$ small enough.
One finally gets from \eqref{eq:estM2} and \eqref{eq:estM1-2} that 
$M_{\vp,j}^1(v)+M_{\vp,j}^2(v)\leq 0$. If \eqref{eq:cos-jv1} does not hold then 
\beq\label{eq:cos-jv2}
\frac{(y_j^Tx_j(v)^{\perp})^2}{\Vert y_j\Vert^2{\Vert x_j(v)\Vert^2}}\geq 1/\sqrt{2}.
\eeq
In that case, 
$$
M_{\vp,j}^1(v)\leq \vp\frac{2\vp C_*^3}{C_0}\frac{S(\Vert x_j(v)\Vert)}{\Vert x_j(v)\Vert}\frac{(y_j^Tx_j(v)^{\perp})^2}{\Vert x_j(v)\Vert^2}
\leq \frac{4\vp C_*^3}{C_0}\frac{S(\Vert x_j(v)\Vert)}{\Vert x_j(v)\Vert}\frac{(y_j^Tx_j(v)^{\perp})^2}{\Vert x_j(v)\Vert^2}.
$$
Adding the above inequality with \eqref{eq:estM2} yields that
$$
M_{\vp,j}^1(v)+M_{\vp,j}^2(v)\leq \Big(\vp C_*^3(2/C_0+1)-1\Big)
\frac{S(\Vert x_j(v)\Vert)}{\Vert x_j(v)\Vert}\frac{(y_j^Tx_j(v)^{\perp})^2}{\Vert x_j(v)\Vert^2}\leq 0,
$$
where the last inequality is obtained for $\vp$ small enough.
The argument for \underline{Case $1$}-Lemma~\ref{le:K} is complete.

\underline{Case $2$.} There exists $\bar{t}\in [0,T]$ such that $\Vert z(\bar{t})-y(\bar{t})\Vert\leq C_*\Vert z(\bar{t})\Vert/2$. One deduces at once that 
$$
\Vert z(\bar{t})\Vert\leq \frac1{\frac{C_*}2-1}\Vert y(\bar{t})\Vert.
$$
For $C_*$ universal constant large enough with respect to 
one, we have that $\Vert y(\bar{t})\Vert\geq R^{1/2}/2$ and 
we can easily rewrite \eqref{eq:est-y1} 
as
\beq\label{eq:est-y3}
(1-4\rho)\Vert y(\bar{t})\Vert \leq \Vert y(t)\Vert \leq (1+4\rho)\Vert y(\bar{t})\Vert,\quad t\in [0,T].
\eeq
By computations similar to those leading to \eqref{eq:est-y2}, one 
gets that there exists a subinterval $I_{bad}$ of $[0,T]$ of length at most $4/C_*$ such that $\Vert z(t)-y(t)\Vert\leq C_*\Vert z(t)\Vert$ for $t\in [0,T]\setminus I_{bad}$. We can therefore subdivide $[0,T]$ in at most three disjoint subintervals, $I_1,I_2$ and $I_{bad}$ such that, if one writes $M_\vp=M_{\vp,1}+M_{\vp,bad}+M_{\vp,2}$ according to the subdivision $[0,T]=I_1\cup I_{bad}\cup I_2$, then 
both $M_{\vp,1}$ and $M_{\vp,2}$ are negative since we can apply to each of them \underline{Case $1$} and one has the direct estimate
$$
M_{\vp,bad}\leq 4|I_{bad}|\Vert y(\bar{t})\Vert\leq 
\frac{16}{C_*}\Vert y_0\Vert\leq
\frac{32}{\rho C_*}T.
$$
By choosing $C_*$ large enough with respect to $\rho$ and $C_1$, one finally obtains \eqref{eq:Kvp-1}.

\end{proof}
\subsection{Proof of Proposition~\ref{th:GOAL-CI}}
This will be obtained in three steps, with the help of Proposition~\ref{prop:key1}. The first step is an easy consequence of Proposition~\ref{prop:key1}.
\begin{lemma}\label{le:limsup}
Consider the constants $\vp_0$ and $R$ defined in Proposition~\ref{prop:key1}. Then, for every $\vp\in (0,\vp_0)$ and every $(z_0,y_0)\in\R^4$, there exists a time 
$T_1(z_0,y_0)$ such that 
\beq\label{eq:key3}
V_0(z(t),y(t))\leq 2R,\quad t\geq T_1(z_0,y_0),
\eeq
where $(z,y)$ denotes the trajectory of $(T_\vp)$ starting at $(z_0,y_0)$.
\end{lemma} 
\begin{proof} First, notice that the inequality 
$V_0(z,y)\geq M$ for $M$ large implies that either $\Vert y\Vert\geq \sqrt{M/2}$ or $\Vert z\Vert\geq M/2$.

We now start the argument of the proposition. Fix $\vp\in (0,\vp_0)$ and $
(z_0,y_0)\in\R^4$ and consider the trajectory $(z,y)$ of $(T_\vp)$ starting at $(z_0,y_0)$. 
Clearly, an immediate argument by contradiction using Proposition~\ref{prop:key1} yields that there exists a time $T_1\geq 0$ such that 
$V_0(z(T_1),y(T_1))\leq R$. One can show the conclusion by taking $T_1=T_1(z_0,y_0)$. Indeed, if it is not possible, then by a obvious continuity argument there exists $T_2>T'_1\geq T_1$ such that 
$$
\frac{3R}2=V_0(z(T'_1),y(T'_1))
\leq V_0(z(t),y(t))\leq 2R=V_0(z(T'_1),y(T'_1))
,\ T'_1\leq t\leq T_2.
$$
Since $\Delta V_0\Big\vert_{T'_1}^{T_2}=R/2$ and $\Vert y(t)\Vert\leq 2R^{1/2}$ on $[T'_1,T_2]$, one deduces that $T_2-T'_1\geq{1/2}$. Applying Proposition~\ref{prop:key1} from $T'_1$ immediately yieds that there exists $t_1\in [T'_1,T_2]$ such that $\Delta V_0\Big\vert_{T'_1}^{t_1}<0$
which is a contradiction. 

\end{proof}
The second step to complete the proof of 
Proposition~\ref{th:GOAL-CI} consists in improving Estimate
\eqref{eq:key2} and get a more precise one with the additional information that trajectories are now universally bounded (i.e., independently of $\vp$) thanks to Lemma~\ref{le:limsup}. We get the following.
\begin{lemma}\label{le:L2gain}
With the above notations, there exist $\vp_0$ and $C_R>0$ such that, 
for every $\vp\in (0,\vp_0)$ and every $(z_0,y_0)\in\R^4$, there exists a time $T_2:=T_2(z_0,y_0)$ for which, for every $T\geq \rho$ with $1/2\leq T\leq 2$ and $T/\vp$ integer, one has
\beq\label{eq:key4}
\Delta V_0\Big\vert_{T_2}^{T_2+T}\leq 
-C_R\int_{T_2}^{T_2+T}\left[\Vert y(t)\Vert^2+(b_\vp^Tz)^2\right]dt.
\eeq
\end{lemma}
\begin{proof}
To proceed, one looks back at the argument of Proposition~\ref{prop:key1} . We can suppose with no loss of generality, in the argument of Proposition~\ref{prop:key1} that $\Vert z(t)\Vert\leq 2R^2$ and $\Vert y(t)\Vert\leq 2R$ for every $t\in [0,T]$ where now $T\geq \rho$ is arbitrary. 
Moreover we will now use the following obvious estimates: there exists two positive constants $C_R^1$ and $C_R^2$ depending only on $R$ such that, for $V_0(z,y)\leq 2R^2$ and $t\in[0,T]$ it holds
\beq\label{eq:obvious}
\Vert \dot z\Vert+\Vert \dot y\Vert\leq C_R^2 \Big(\Vert y\Vert+\Vert z\Vert\Big),\
C_R^1\Vert (z,y)\Vert^2\leq V_0(z,y)\leq C_R^2\Vert (z,y)\Vert^2,\
\vert \dot V_0\Vert\leq C_R^2V_0.
\eeq
We now follow again the proof of Proposition~\ref{prop:key1} with the objective of providing better estimates  of all the $\vp TO(1)$ that have appeared. We start by choosing 
$T$  so that  $1/2\leq T\leq 2$ and $T/\vp$ integer. In that way, we have eliminated the error term occuring when one performs the the change of times $s=j+v$ in each interval $[j,j+1]$, for $0\leq j\leq k-1$ to pass from \eqref{eq:K1vp} and \eqref{eq:K2vp} to \eqref{eq:Kvp-2}. 

The next error terms to handle are those occuring in \eqref{eq:errorKj1} and then in
\eqref{eq:errorKj2}. Those occuring in \eqref{eq:errorKj1} can  now be replaced by
$$
C_R\vp\Big(\int_0^1(b_{1}^T z)^2dv\Big)^{1/2}\Big(\int_0^1\Vert y\Vert^2dv\Big)^{1/2},
$$
by using systematically Cauchy-Schwarz inequality and the global Lipschitz character of $f$. 
By plugging the factor $\vp$ inside the integrals, coming back to the time scale 
$t\in [0,T]$ and then summing up with respect to $0\leq j\leq k-1$, we can bound the error term in \eqref{eq:errorKj2} as
$$
2C_R\vp\sum_{j=0}^{k-1}\Big(\int_{Tj/k}^{T(j+1)/k)}(b_{\vp}^T z)^2dt\Big)^{1/2}\Big(\int_{Tj/k}^{T(j+1)/k)}\Vert y\Vert^2dt\Big)^{1/2},
$$
where $C_R$ is a positive constant only depending on $R$. This is then trivially smaller than 
\beq\label{eq:errorKey}
C_R\vp\Big(\int_{0}^{T}(b_{\vp}^T z)^2dt+\int_{0}^{T}\Vert y\Vert^2dt\Big).
\eeq
All the other error terms $\vp TO(1)$ can bounded in a similar way together with the fact that there exists some positive constant $C_R$ only depending on $R$, for every $0\leq j\leq k-1$,
$$
\Vert y_i\Vert \leq C_R\Big(\int_0^1\Vert y\Vert^2dv\Big)^{1/2},\quad
\Vert z_i\Vert^2 \leq C_R\Big(\int_0^1(b_{1}^T z)^2dv+\int_0^1\Vert y\Vert^2dv\Big).
$$
This follows simply from the left part of \eqref{eq:obvious}. 

On the other hand, it is immediate that one can improve the estimates in \eqref{eq:Mj1} and \eqref{eq:Mj2} to derive that $M_{\vp,j}^1(v)+M_{\vp,j}^1(v)$ are upper bounded by 
$-C_R\Vert y_j\Vert^2$, which implies that 
for $0\leq j\leq k-1$, one has
$$
M_{\vp,j}\leq -C_R\Vert y_j\Vert^2,
$$
for some positive constant $C_R$ only depending on $R$.
Hence the quantity $M_\vp$ defined in \eqref{eq:Mvp} is upper bounded as
$$
M_\vp\leq -C_R\vp\sum_{j=0}^{k-1}\Vert y_j\Vert^2.
$$
After coming back to the time scale $t\in[0,T]$, one easily recognizes that the right-hand side of the above inequality is a Riemann sum of the function $t\mapsto \Vert y(t)\Vert^2$. Since it has a derivative bounded by some positive constant only depending on $R$, 
one gets that 
$$
M_\vp\leq -C_R \int_0^{T}\Vert y\Vert^2dt.
$$
Gathering all the above estimates and eventually diminishing $\vp_0$ finally yields that 
\beq\label{eq:estK2}
K_\vp\leq C_R^1\vp\int_{0}^{T}(b_{\vp}^T z)^2dt-C_R^2\int_0^{T}\Vert y\Vert^2dt,
\eeq
for some positive constants $C_R^1,C_R^2$ only depending on $R$.

On the other hand, one has by using \eqref{eq:key3} that there exists a positive constant $C_R>0$ such that 
\beq\label{eq:estL2}
L_\vp=-\int_0^{T}\frac{S(\Vert z\Vert)}{\Vert z\Vert}
\frac{\sigma(b_\vp^Tz)}{b_\vp^Tz}(b_\vp^Tz)^2dt\leq -C_R\int_0^{T}(b_\vp^Tz)^2dt.
\eeq
By collecting \eqref{eq:estK2} and \eqref{eq:estL2}, we deduce 
\eqref{eq:key4}.

\end{proof}
The final step of the argument takes advantage of the previous 
estimate. 
We obtain from the above that the $L^2$-norm of $b_\vp^Tz$ over $\mathbb{R}_+$ is finite and, since the time derivative of $b_\vp^Tz$ is bounded (with a constant depending on $\vp$), we deduce that $b_\vp^Tz$ tends to zero at $t$ tends to infinity by Barbalat Lemma. Recall now that $b_\vp^Tz=K_\vp^Tx$ as the latter term appears in $(S_\vp)$ defined in \eqref{eq:CDI-Svp} with the choice of $K_\vp$ made in \eqref{eq:bKvp}. Then we have that 
$K_\vp^Tx$ tends to zero at $t$ tends to infinity. One can therefore rewrite $(S_\vp)$ as $\dot x=A_\vp x+f(t)b$, where
\beq\label{eq:Avp-b}
A_\vp=J_2(\frac{2\pi}\vp)-bb^T,\quad
b=\bpm 0\\e_2\epm,\quad
f(t)=K_\vp^Tx-\sigma(K_\vp^Tx).
\eeq
Since $A_\vp$ is Hurwitz and $f$ tends to zero at $t$ tends to infinity, one concludes that any trajectory $x$ of $(S_\vp)$ converges to zero as $t$ tends to infinity. Since $f(t)$ is actually a $o(\Vert x\Vert)$, as $\Vert x\Vert$ tends to zero, one also gets that $(S_\vp)$ is locally exponentially stable with respect to the origin.

The proof of 
Proposition~\ref{th:GOAL-CI} is complete.

\section{Appendix}\label{se:app}
In this section, we collect technical results used throughout the paper.

We start by providing an argument for the items in Remark~\ref{rem:sat}. Item $(c1)$ is immediate. For Item $(c2)$, it is enough to prove the statements for $\xi>0$ and conclude by continuity. The first part of that item follows from the fact that $\sigma'$ is decreasing over $\mathbb{R}_+$ and the inequality
\beq\label{eq:sat00}
\sigma(\xi)=\int_0^{\xi}\sigma'(v)dv\geq \xi\sigma'(\xi),\quad \forall \xi\geq 0.
\eeq
The second statement is a consequence of the following equality
$$
\Big(\frac{\sigma(\xi)}{\xi}\Big)'=\frac{\xi\sigma'(\xi)-\sigma(\xi)}{\xi^2},\quad \forall \xi\neq 0,
$$
and \eqref{eq:sat00}. 
As for Item $(c3)$, notice that for every $\xi\neq 0$,
$$
\sigma'(\xi)\leq 2\frac{\sigma(\xi)}{\xi}-\frac{\sigma(\xi/2)}{\xi/2}=\frac{\sigma(\xi)-\sigma(\xi/2)}{\xi/2}=
\int_{\xi/2}^{\xi}\sigma'(s)\, ds\leq \sigma'(\xi/2).
$$
By letting $\xi$ tend to zero, one gets that
$$
\limsup_{\xi\to 0}\sigma'(\xi)\leq \sigma'(0)\leq \liminf_{\xi\to 0}\sigma'(\xi),
$$
hence the last part of  Item $(c3)$.

We now prove the following results on the modified saturation function.
\begin{proposition}[Modified saturation function associated with a saturation function $\sigma$]\label{prop:mod-sat}
The modified saturation function $S:\R\to\R$ associated with the saturation function $\sigma$ is the function defined on $\R$ by
\beq\label{eq:mod-sat}
 S(\xi)=\int_0^{1}s_{2\pi v}\sigma(\xi s_{2\pi v})dv=\frac2{\pi}\int_0^{\pi/2}s_v\sigma(\xi s_v)dv.
 \eeq
 Then $S$ has the following properties.
 \begin{description}
 \item[$(S1)$] One has the following expressions for $S'$: 
  \beq\label{eq:S'-1}
 S'(\xi)=\frac2{\pi}\int_0^{\pi/2}\sigma'(\xi s_v)s^2_vdv,
 \eeq
and, for $\xi\neq 0$, 
 \beq\label{eq:S'-33}
 S'(\xi)=\frac2{\pi}\int_0^{\pi/2}\frac{\sigma(\xi)-\sigma(\xi s_v)}{\xi(1-s_v)}h(v)dv,
 \eeq
where $h:[0,\pi/2]\to \R_+$ is the continuous function defined by $g(v)=\frac{1-s_v}{c^2_v}s_v(1+c^2_v)$ for $v\in [0,\pi/2)$ and $h(\pi/2)=1/2$. As a consequence,  $S$ is a saturation function of class $C^1$ with $S_\infty=\sigma_\infty/2$, $S'(0)=\sigma'(0)/2$ and $S'>0$;
 \item[$(S2)$] there exists $C_2>0$ such that, for every $\xi\in\mathbb{R}$ and $M\geq 1$, one has
\beq\label{eq:S'-4}
S'(M\xi)\geq \frac{C_2}{M^3}S'(\xi).
 \quad \xi\in\mathbb{R}.
 \eeq
\end{description}
\end{proposition}
\begin{proof} The definition of $S$ shows that it is positive and bounded. Eq.~\eqref{eq:S'-1} is immediate, which implies that $S'(0)=\sigma'(0)/2$, $S'$ is bounded, positive and non increasing. Moreover, \eqref{eq:S'-33} implies that $S'$ is continuous for $\xi\neq 0$. It remains to show the continuity of $S'$ at $\xi=0$. For that purpose, note that from \eqref{eq:S'-1} one has
$$
S'(\xi)-S'(0)=\frac2{\pi}\int_0^{\pi/2}\Big(\sigma'(\xi s_v)-\sigma'(0)\Big)dv.
$$
Continuity of $\sigma'$ at $\xi=0$ immediately implies continuity of $S'$ at $\xi=0$.
 
 As for Item $(S2)$, one can assume $\xi>0$ with no loss of generality. From \eqref{eq:S'-1} and the facts that $\sigma'$ is decreasing and $2v/\pi\leq s_v\leq v$ for $v\in[0,2\pi]$, one deduces that there exists two universal constants $C_1,C_2>0$ such that, 
$$
\frac{C_1}{\xi^3}\int_0^{\xi}\sigma'(v)v^2dv\leq
\frac{C_1}{\xi^3}\int_0^{\pi\xi/2}\sigma'(v)v^2dv\leq S'(\xi)\leq
\frac{C_2}{\xi^3}\int_0^{\xi}\sigma'(v)v^2dv,\quad\forall \xi>0.
$$
For $\xi>0$, set $H(\xi)=\int_0^{\xi}\sigma'(v)v^2dv$, which is an increasing function. For $M\geq 1$, one gets 
$$
M^3\xi^3S'(M\xi)\geq C_1H(M\xi)\geq C_1H(\xi)\geq \frac{C_1}{C_2}\xi^3S'(\xi),
$$
from which the conclusion follows. 

 \end{proof}

\end{document}